\documentclass[letterpaper, 10 pt, conference]{ieeeconf}  

\IEEEoverridecommandlockouts                              

\usepackage{xspace}        
\usepackage{mathtools}
\usepackage[hidelinks]{hyperref}
\usepackage{bbm}
\usepackage{url}
\usepackage{cleveref}
\usepackage{dsfont} 
\usepackage{siunitx}
\usepackage{booktabs}       
\usepackage{multirow}
\usepackage{algorithmic}
\usepackage[ruled,vlined]{algorithm2e}       
\usepackage{amsthm,amsmath,amssymb}
\usepackage{color}
\usepackage{comment}
\usepackage{tikz}
\usepackage{pgfplots}
\usepackage{bbm}
\usepackage{cite}
\usepackage{graphicx}
\usepackage{textcomp}
\usepackage{epstopdf}
\usepackage{bm}
\usepackage{color}
\usepackage{mathrsfs}
\usepackage{tabu}

\newtheorem{lemma}{Lemma}
\newtheorem{theorem}{Theorem}
\newtheorem{assumption}{Assumption}

\begin{document}
%
\title{Learning of Nash Equilibria in Risk-Averse Games}
%
%
%

\author{Zifan Wang, Yi Shen, Michael M. Zavlanos, and Karl H. Johansson
\thanks{* This work was supported in part by Swedish Research Council Distinguished Professor Grant 2017-01078, Knut and Alice Wallenberg Foundation, Wallenberg Scholar Grant, the Swedish Strategic Research Foundation CLAS Grant RIT17-0046, AFOSR under award \#FA9550-19-1-0169, and  NSF under award CNS-1932011. 
}
\thanks{Zifan Wang and Karl H. Johansson are with Division of Decision and Control Systems, School of Electrical Enginnering and Computer Science, KTH Royal Institute of Technology, and also with Digital Futures, SE-10044 Stockholm, Sweden. Email: \{zifanw,kallej\}@kth.se.}
\thanks{Yi Shen and Michael M. Zavlanos are with the Department of Mechanical Engineering and Materials Science, Duke University, Durham, NC, USA. Email: \{yi.shen478, michael.zavlanos\}@duke.edu}%
}

\maketitle

\begin{abstract}
This paper considers risk-averse learning in convex games involving multiple agents that aim to minimize their individual risk of incurring significantly high costs.
Specifically, the agents adopt the conditional value at risk (CVaR) as a risk measure with possibly different risk levels.
To solve this problem, we propose a first-order risk-averse leaning algorithm, in which the CVaR gradient estimate depends on an estimate of the Value at Risk (VaR) value combined with the gradient of the stochastic cost function. 
Although estimation of the CVaR gradients using finitely many samples is generally biased, we show that the accumulated error of the CVaR gradient estimates is bounded with high probability.
Moreover, assuming that the risk-averse game is strongly monotone, we show that the proposed algorithm converges to the risk-averse Nash equilibrium.
We present numerical experiments on a Cournot game example to illustrate the performance of the proposed method.
\end{abstract}


%
\IEEEpeerreviewmaketitle

\section{Introduction}

Convex games find applications in many domains ranging from online marketing \cite{lin2021doubly} to transportation networks \cite{sessa2019no}.
In these convex games, agents simultaneously take actions to minimize their loss functions, which are influenced by the actions taken by other agents.
The concept of a Nash equilibrium is central in the analysis of such games and represents a stationary point from which no agent has an incentive to deviate, see, e.g., \cite{tatarenko2018learning,mertikopoulos2019learning,bravo2018bandit,drusvyatskiy2022improved,lin2020finite,narang2022learning}.

An important consideration, especially in high-stakes applications, is sensitivity of the learned decisions to the possible risks due to the presence of uncertainty \cite{ahmadi2021risk,markowitz1991foundations,cardoso2019risk,kalogerias2022fast}.
For example, in portfolio management \cite{markowitz1991foundations}, constructing a risk-averse portfolio rather than one that yields the highest expected return in a market with much uncertainty is preferred, since it reduces the risk of suffering from large losses.
Additionally, in clinical trials \cite{cardoso2019risk}, a drug that has high average performance and high probability of negative effects may not be desirable due to the safety critical nature of this application.
The key idea in risk-averse learning is to replace the expectation in the objective function by a more general objective that employs measures of risk and considers the whole distribution of the stochastic cost. 
Popular risk measures include mean-deviation functionals \cite{shapiro2021lectures}, Value at risk (VaR) \cite{szorenyi2015qualitative,zhang2021quantile,kolla2019concentration} and Conditional Value at Risk (CVaR) \cite{rockafellar2000optimization,tamkin2019distributionally,you2021gaussian}.

In this paper, we consider learning in risk-averse games with continuous action sets in which agents aim to minimize their risk-averse cost functions.
Specifically, we employ CVaR as the risk measure and provide a sufficient condition that guarantees the strong monotonicity for the risk-averse game. We assume that the agents are capable of observing other agents’ actions and computing the gradient of their own stochastic cost function. Although computing the CVaR gradient is usually computationally difficult, we show that the CVaR gradient can be expressed as a function of the VaR value and the gradient of the stochastic cost function. Building on this insight, we develop VaR estimates using all historical samples and then use these VaR estimates to construct the CVaR gradient estimates. 
By constructing an upper confidence bound for the VaR estimate errors, we show that the proposed algorithm converges to the risk-averse Nash equilibrium with high probability under the established strong monotonicity condition.
Finally, we present numerical experiments on a Cournot game to verify our results.

To the best of our knowledge, the analysis of risk-averse games is underexplored in the literature,  with a few exceptions
\cite{yekkehkhany2020risk,slumbers2022learning,wang2022risk,wang2022zeroth}.
Specifically, \cite{yekkehkhany2020risk} considers agents that aim to maximize the probability of receiving maximum reward instead of the expected payoff, and shows that risk-averse Nash equilibria always exist. 
However, specific algorithms that converge to these risk-averse equilibria are not proposed.
The authors in \cite{slumbers2022learning} define a new concept of risk-averse equilibria in finite games using the mean-variance as the risk measure and prove that such equilibria exist. To find these equilibria, a fictitious play algorithm is proposed, but no theoretical convergence analysis is provided.
The works in \cite{wang2022risk,wang2022zeroth} investigate risk-averse games with continuous action sets and propose several risk-averse learning algorithms, which all rely on one-point zeroth-order optimization and achieve no-regret learning with high probability.
These works focus on the regret analysis but not on the Nash equilibrium convergence considered here.


The rest of the paper is organized as follows. Section~\ref{sec:1_prelim} introduces preliminaries about convex games and the VaR and CVaR risk measures. In Section~\ref{sec:2_problem}, we formally define the risk-averse games and establish the strong monotonicity for such games. In Section~\ref{sec:FO}, we propose a first-order risk-averse learning algorithm and analyze its convergence rate. The performance of the proposed method is illustrated in Section~\ref{sec:Simulation} via an online market problem.
Finally, we conclude the paper in Section~\ref{sec:Conclusion}.

\section{Preliminaries}
\label{sec:1_prelim}
\subsection{Convex Games} 
Consider a repeated game $\mathcal{G}$  with $N$ agents, whose goal is to learn their best individual actions that minimize their local loss functions. At each episode, each agent selects an action $x_i$ from the convex set $\mathcal{X}_i \subseteq \mathbb{R}^{d}$ and receives a cost value of $J_i(x_i,x_{-i}): \mathcal{X} \rightarrow \mathbb{R}$, where $x_{-i}$ represents all agents' actions except for agent $i$, and $\mathcal{X}=\Pi_{i=1}^N\mathcal{X}_i$ is the joint action space. 
For ease of notation, we usually collect all agents' actions in a vector $x:=(x_1,\ldots,x_N)$. 
The game is formally defined as
\begin{align}\label{def:game}
    \mathop{{\rm{min}}}_{x_i \in \mathcal{X}_i} J_i(x_i,x_{-i}).
\end{align}
We call the game \eqref{def:game} a convex game if $J_i(x_i,x_{-i})$ is convex in $x_i$ for all $x_{-i} \in \mathcal{X}_{-i}$, where $\mathcal{X}_{-i}=\Pi_{j \neq i}\mathcal{X}_j$.
As shown in \cite{rosen1965existence}, a convex game always has at least one Nash equilibrium. We denote by $x^{*}$ a Nash equilibrium of the game \eqref{def:game} and this point satisfies 
$J_i(x^{*})\leq J_i(x_i,x_{-i}^{*})$, 
for all $x_i \in \mathcal{X}_i$, $i=1,\ldots,N$.
At this Nash equilibrium point, agents are strategically stable in the sense that each agent lacks incentives to change its action.
Since the loss functions of the agents are convex, the Nash equilibrium can also be characterized by the first-order optimality condition, that is, $\langle \nabla_{x_i} J_i(x^{*}), x_i - x_i^{*} \rangle \geq 0, \; \forall x_i \in \mathcal{X}_i$, where $\nabla_{x_i} J_i(x)$ is the partial derivative of the loss function of agent $i$ with respect to his own action.  We write $\nabla_{i} J_i(x)$ instead of $\nabla_{x_i} J_i(x)$ whenever it is clear from the context. 
In general, it is not easy to show convergence to a fixed Nash equilibrium for games with multiple Nash Equilibria. 
For this reason, recent studies focus on games that are so-called strongly monotone and are well-known to have a unique Nash equilibrium \cite{rosen1965existence}. 
The game~\eqref{def:game} is said to be $m$-strongly monotone if for all $x,x'\in \mathcal{X}$, we have
\begin{align*}
    \sum_{i=1}^N \langle \nabla_i J_i(x) -\nabla_i J_i(x'),x_i-x_i' \rangle \geq m \left\|x -x' \right\|^2.
\end{align*}
Throughout this paper, we use $\left\|\cdot\right\|$ to denote the 2-norm for a vector.

\subsection{CVaR and VaR}
Conditional Value at Risk (CVaR) is a coherent risk measure that satisfies properties of monotonicity, sub-additivity, homogeneity, and translational invariance; see \cite{huang2021off,shapiro2021lectures}.
Formally, for a random variable $Z$ with the cumulative distribution function (CDF) denoted by $F$ and a risk level $\alpha \in (0,1]$, the ${\rm{CVaR}}$ value is defined as ${\rm{CVaR}}_{\alpha}[Z] = \mathbb{E}_F[Z| Z>  {\rm{VaR}}_{\alpha}[Z] ]$, where ${\rm{VaR}}_{\alpha}[Z] = {\rm{inf}}\{ y: F_Z(y) \geq 1-\alpha\}$ is the $1-\alpha$ quantile of the distribution of the random variable $Z$, also known as the Value at Risk (VaR). Intuitively, CVaR represents the average of the worst $\alpha \times 100\%$ cost. If the risk level is selected as $\alpha=1$, it becomes equivalent to the expected (risk-neutral) case. Thus, CVaR has a naturally distributional robust optimization formulation. 
The work \cite{rockafellar2000optimization} introduces a new approach to compute the CVaR value, i.e.,
\begin{align}\label{eq:CVaR:L:func}
    {\rm{CVaR}}_{\alpha}[Z] =  \mathop{{\rm{min}}}_{\nu \in \mathbb{R}} \left\{ \nu + \frac{1}{\alpha} \mathbb{E}_F[Z-\nu ]_{+} \right\},
\end{align}
where $[x]_{+} = {\rm{max}}\{x,0\}$. Besides, \cite{rockafellar2000optimization} shows that the right hand side of \eqref{eq:CVaR:L:func} takes the minimal value when $\nu = {\rm{VaR}}_{\alpha}[Z]$.

\section{Risk-Averse Convex Games}\label{sec:2_problem}
In this section, we formally define the risk-averse game under consideration, using the CVaR risk measure.
Moreover, we establish a condition under which this game is strongly monotone, which lays the theoretical foundation for the subsequent analysis of the Nash equilibrium convergence.

\subsection{Problem Formulation}

We consider convex games with stochastic costs $J_i(x_i,x_{-i},\xi_i): \mathcal{X} \times \Xi_i \rightarrow \mathbb{R}$, where $x_{-i}$ are the actions of all agents except for agent $i$ and $\xi_i \in \Xi_i \subset \mathbb{R}^{n_{\xi}}$ characterizes the uncertainty associated with the cost function, following some static and unknown distribution. We sometimes instead write $J_i(x,\xi_i)$ for ease of notation, where $x=(x_i,x_{-i})$ is the concatenated vector of all agents' actions.
We assume that the diameter of the convex set $\mathcal{X}_i$ is bounded by $D$ for all $i=1,\ldots,N$.

The goal of the risk-averse agents is to minimize the CVaR value of the stochastic cost.
We denote by $F_{i,x}(y)=\mathbb{P}\{J_i(x,\xi_i)\leq y \}$ the CDF of the random cost $J_i(x,\xi_i)$ and $J^{\alpha_i}$ the $1-\alpha_i$ quantile of this distribution. 
Then, for a given risk level $\alpha_i\in (0,1]$, the CVaR of the cost function $J_i(x,\xi_i)$ is defined as 
\begin{align}\label{eq:CVaR:def} 
    C_i(x):& ={\rm{CVaR}}_{\alpha_i}[J_i(x,\xi_i)] \nonumber \\
    :& = \mathbb{E}_{F_{i,x}}[J_i(x,\xi_i)|J_i(x,\xi_i)\geq J^{\alpha_i}].
\end{align} 
Notice that the CVaR value is determined by the distribution function $F_{i,x}(y)$ for a given $\alpha_i$, so we sometimes write CVaR as a function of the distribution function, i.e.,  ${\rm{CVaR}}_{\alpha_i}[F_{i,x}]:={\rm{CVaR}}_{\alpha_i}[J_i(x,\xi_i)] $.
%
%
Given different risk levels $\alpha_i$, $i=1,\ldots,N$, the goal of each agent is to minimize the individual risk-averse loss function, i.e.,
\begin{align}\label{def:risk-averse_game}
    \mathop{{\rm{min}}}_{x_i \in \mathcal{X}_i} C_i(x_i,x_{-i}).
\end{align}
In this paper, we assume that the function $J_i(x_i,x_{-i},\xi_i)$ is convex in $x_i$ for every $x_{-i} \in \mathcal{X}_{-i}$ and $\xi_i \in \Xi_i$, for all $i=1,\ldots,N$. 
By virtue of \eqref{eq:CVaR:L:func}, we have $C_i(x) = \nu_i^{*} + \frac{1}{\alpha} \mathbb{E}[J_i(x,\xi_i)-\nu_i^{*} ]_{+}$, where $\nu_i^{*} = {\rm{VaR}}_{\alpha_i}[J_i(x,\xi_i)]$. Since both the operations of point-wise supremum over convex functions and expectation preserve convexity, we can obtain that $C_i(x)$ is convex in $x_i$ for all $i=1,\ldots,N$.
As a result, the risk-averse game \eqref{def:risk-averse_game} is a convex game and has at least one equilibrium point by Theorem 1 in \cite{rosen1965existence}.

We consider the game setting in which the agents compete in an open environment that discloses all agents' past actions publicly. Specifically, each agent selects an action at each episode and at the next episode the previous actions will be shared to all agents. Besides, we assume that the agents can sample the static uncertainty in the game and know the form of its own cost function $J_i$. Thus, each agent can obtain the gradient information of $J_i$. 
The goal of this paper is to analyze the CVaR gradient and then design a risk-averse learning algorithm for the risk-averse game \eqref{def:risk-averse_game}  that achieves Nash equilibrium convergence.

\subsection{Strong Monotonicity Analysis}
In this section, we provide a sufficient condition that establishes strong monotonicity for the risk-averse game \eqref{def:risk-averse_game}.
We note that if the risk-neutral game is monotone, where agents aim to minimize the expected cost values, the risk-averse game may not necessarily be monotone. We provide an example in the following to illustrate this point. In this example, we show that when a risk-neutral game is strongly monotone and has a unique Nash equilibrium, the risk-averse game may have infinitely many Nash equilibria. 

\noindent \textbf{Example.} Consider a game with two agents in which each agent has the cost function $J_i(x,\xi_i) = c+ax_i^2+ax_i x_{-i} - abx_i + \frac{4a}{3d}x_ix_{-i}\xi_i$, $i=1,2$, where $a>0$ and $\xi_i \sim U(0,d)$.
In the risk-neutral game, each agent aims to minimize the expected cost values $\pi_i(x) =\mathbb{E}_{\xi_i} [J_i(x,\xi_i)] $, while in the risk-averse game, each agent aims to minimize $C_i(x) = {\rm{CVaR}}_{0.5}[J_i(x,\xi_i)]$.
It is easily verified that the risk-neutral game satisfies the monotone condition
\begin{align*}
    \sum_{i} \langle \nabla_i \pi_i(x) - \nabla_i \pi_i(x'),x_i- x_i' \rangle \geq m_0 \left\|x-x'\right\|^2,
\end{align*}
with $m_0 = a$.
However, in the risk-averse case, we have $C_i(x) = c+ax_i^2+2ax_i x_{-i} -abx_i$ and thus $\nabla_i C_i(x) = 2ax_i +2a x_{-i} -ab$. 
Since Nash equilibria are the points that satisfy $ \nabla_1 C_1(x) = \nabla_2 C_2(x)=0$, which correspond to all the points located on the line $x_1+x_{2}=\frac{b}{2}$, it follows that the risk-averse game has an infinite number of Nash equilibria.

Note that we require  that the risk-averse game is strongly monotone regardless of the choices of $\alpha_i \in (0,1]$, thus, every quantile of the distribution of the stochastic cost is expected to satisfy the strong monotonicity condition.
To establish the strong monotonicity condition of the risk-averse game \eqref{def:risk-averse_game} for all the choices of $\alpha_i$, we make the following assumption on the stochastic cost. 
\begin{assumption}\label{assump:strong_monotone}
For each agent $i$, the cost function can be decomposed as  $J_i(x,\xi_i)=f_i(x_i,\xi_i)+g_i(x)$.
Moreover, $f_i(x_i,\xi_i)$ is absolutely continuous and its VaR value is differentiable for any $x_i \in \mathcal{X}_i$.
Besides, $f_i(x_i,\xi_i)$ is convex in $x_i$, for every $\xi_i \in \Xi_i$ and  $g_i(x)$ satisfies 
\begin{align*}
    \sum_i \langle \nabla_i g_i(x) - \nabla_i g_i(x'),x_i- x_i' \rangle \geq m \left\| x-x'\right\|^2,
\end{align*}
for all $x,x' \in \mathcal{X}$.
\end{assumption}
Assumption \ref{assump:strong_monotone} states that the uncertainty in the stochastic cost of each agent does not rely on other agents' actions. This assumption holds for many classes of games including Cournot games. For example, in a market consisting of multiple competitive companies that produce goods at different levels, the price of these goods consists of the deterministic term $P(x)$ that depends on the total amount of production and the random term $\xi_i$ that captures the randomness in the market induced by, e.g., government policies or the climate changes. This randomness usually does not depend on  the production levels. As a result, the reward of the $i$-th company can be given by $r_i(x)=-J_i(x)=(P(x)+\xi_i)x_i$.

Given Assumption \ref{assump:strong_monotone}, we show that the risk-averse game is strongly monotone in the following lemma.
\begin{lemma}\label{lemma:strongly_monotone}
Let Assumption \ref{assump:strong_monotone} hold. Then, the risk-averse game \eqref{def:risk-averse_game} is $m$-strongly monotone.
\end{lemma}
\begin{proof}
Given Assumption \ref{assump:strong_monotone} and the translation invariance property of CVaR, we have that $C_i(x) = {\rm{CVaR}}_{\alpha_i}[J_i(x,\xi_i)] = {\rm{CVaR}}_{\alpha_i} [ f_i(x_i, \xi_i)] +g_i(x)$. Define ${\rm{CVaR}}_{\alpha_i} [ f_i(x_i, \xi_i)] = h_i(x_i)$. Since $f_i$ is convex in $x_i$, by Lemma 3 in \cite{cardoso2019risk}, $h_i(x_i)$ is convex in $x_i$. Then, it holds that
\begin{align}
    & \sum_i \langle \nabla_i C_i(x) - \nabla_i C_i(x'),x_i- x_i' \rangle   \nonumber \\
    = & \sum_{i} \langle \nabla_i h_i(x_i) - \nabla_i h_i(x_i'),x_i- x_i' \rangle \nonumber \\
    &+ \sum_i \langle \nabla_i g_i(x) -\nabla_i g_i(x') ,x_i- x_i' \rangle \nonumber \\
    \geq &  \sum_i \langle \nabla_i g_i(x) -\nabla_i g_i(x') ,x_i- x_i' \rangle  
    \geq  m \left\|x-x'\right\|^2, \nonumber 
\end{align}
for all $ x,x' \in \mathcal{X}$ and $\alpha_i \in(0,1]$, $i=1,\ldots,N$. The first inequality follows from the convexity of $h_i$.
The proof is complete.
\end{proof}
Given that the risk-averse game is strongly monotone, the risk-avesrse Nash equilibrium $x^{*}$ is unique \cite{rosen1965existence}.
In what follows, we require the following  assumption on the cost function $J_i$.
\begin{assumption}\label{assump:grad:bound}
$\left\|\nabla_i J_i(x,\xi_i)\right\| \leq B$, for all $i=1,\ldots,N$.
\end{assumption}
This assumption is common in the literature and holds in many applications, e.g., Cournot Games and Kelly auctions; see \cite{bravo2018bandit,lin2021doubly}. 

\section{A Risk-Averse Learning Algorithm}\label{sec:FO}
In this section, we propose a method that enables each agent to use the gradient of the stochastic cost to minimize the risk-averse cost function.

\subsection{Algorithm Description}
Before presenting the designed algorithm, we analyze the expression of the CVaR gradient.
To do so, we define the auxiliary function
\begin{align*}
    L_i(x,\nu_i) := \nu_i + \frac{1}{\alpha_i} \mathbb{E}_{\xi_i}[J_i(x,\xi_i) - \nu_i]_{+},
\end{align*}
where $\nu_i \in \mathbb{R}$ is an auxiliary variable. 
We assume that $J_i(\cdot,\xi_i)$ is Lipschitz and diffentiable for every $\xi_i$. Then, it is shown in \cite{kalogerias2022fast} that the (sub)gradient of $L_i$ can be represented as 
\begin{align*}
    \nabla_iL_i(x,\nu_i)= \mathbb{E}_{\xi_i}  \left[ \begin{array}{c} \frac{1}{\alpha_i} \textbf{1}\left\{ J_i(x,\xi_i)\geq \nu_i) \right\} \nabla_i J_i(x,\xi_i) \\ 1 - \frac{1}{\alpha_i} \textbf{1}\left\{ J_i(x,\xi_i)\geq \nu_i) \right\} \end{array} \right],
\end{align*}
where $\nabla_i J_i(x,\xi_i)$ represents the derivative of the function $J_i(x,\xi_i)$ with respect to $x_i$, $\nabla_iL_i(x,\nu_i)$ denotes the derivative of the function $L_i$ with respect to $(x_i,\nu_i)$, and $\textbf{1}\{ \cdot\}$ is the indicator function.
As shown in \cite{rockafellar2000optimization}, the VaR value of the random cost $J_i(x,\xi_i)$, which we denote by $\nu_i^{*}(x)$, satisfies $C_i(x) = {\rm{CVaR}}_{\alpha_i}[J_i(x,\xi_i)] = L_i(x,\nu_i^{*}(x) )$ and $\nu_i^{*}(x) = {\rm{left \; endpoint \; of}} \; \mathcal{A}_i^{*}(x) $, where the set $\mathcal{A}_i^{*}(x) : =  {\rm{argmin}}_{\nu_i} L_i(x,\nu_i)$.

In the following lemma, we connect the CVaR gradient with the gradient of the function $L_i(x,\nu_i)$.
\begin{lemma}\label{lemma:L:CVaR:gradient}
    It holds that 
\begin{align*}
    \nabla_i C_i(x) =&\nabla_{x_i} L_i(x,\nu_i) \Big| _{\nu_i= \nu_i^{*}(x)} \\
    =&\mathbb{E}_{\xi_i}\left[\frac{1}{\alpha_i} \textbf{1}\left\{ J_i(x,\xi_i)\geq \nu_i^{*}(x)) \right\} \nabla_i J_i(x,\xi_i) \right].
\end{align*}
\end{lemma}
\begin{proof}
Since $L_i(x,\nu_i)$ is convex in $\nu_i$, $\nu_i^{*}(x)$ satisfies that $\nabla_{\nu_i} L_i(x,\nu_i^{*}(x)) (\nu_i' - \nu_i^{*}(x)) \geq 0$, for all $\nu_i' \in \mathbb{R}$. Due to the fact that the value of $\nu_i^{*}(x)$ is bounded, we conclude that $\nabla_{\nu_i} L_i(x,\nu_i )\Big|_{\nu_i=\nu_i^{*}(x)}=1 - \frac{1}{\alpha_i} \textbf{1}\left\{ J_i(x,\xi_i)\geq \nu_i^{*}(x) \right\}=0$.
Hence, we have 
\begin{align*}
    &\nabla_i C_i(x)  = \nabla_{x_i} L_i(x,\nu_i^{*}(x)) \nonumber \\ 
    =& \nabla_{x_i} L_i(x,\nu_i) \Big| _{\nu_i= \nu_i^{*}(x)} +\nabla_{\nu_i} L_i(x,\nu_i )\nabla_{x_i} \nu_i^{*}(x)   \Big|_{\nu_i=\nu_i^{*}(x)}  \nonumber \\
    = & \nabla_{x_i} L_i(x,\nu_i) \Big| _{\nu_i= \nu_i^{*}(x)},
\end{align*}
which completes the proof.
\end{proof}

\begin{algorithm}[t]
\caption{First-order risk-averse learning } \label{alg:algorithm:FO}
\begin{algorithmic}[1]
    \REQUIRE Initial value $x_0$, step size $\eta$, the total number of episodes $T$, risk level $\alpha_i$, $i=1,\cdots,N$.
    \FOR{$episode \; t=1,\ldots,T$}
        \STATE Each agent plays $x_{i,t} $, $i=1,\ldots, N$
        \FOR{agent $ i=1,\ldots,N$}
        \STATE Each agent samples $\xi_{i}^t$
        \STATE Each agent computes 
        cost evaluations $J_i(x_t,\xi_i^k)$ and $\nabla_i J_i(x_t,\xi_i^k)$, $k=1,\ldots,t$ 
        \STATE Build EDF by \eqref{eq:edf_FO} 
        \STATE Compute VaR estimate by \eqref{eq:FO:var:estimate:G_hat}
        \STATE Construct gradient estimate by \eqref{eq:FO:grad}
        \STATE Update action $x_{i,t+1}$ by \eqref{eq:FO:update:action}
        \ENDFOR
    \ENDFOR
\end{algorithmic}
\end{algorithm}
Lemma \ref{lemma:L:CVaR:gradient} provides an alternative way to compute the CVaR gradient by means of computing the gradient $\nabla_{x_i} L_i(x,\nu_i) $ and the VaR value $ \nu_i^{*}(x)$. 
Similar arguments can also be found in \cite{hong2009simulating}.
Leveraging this result, we propose a first-order risk-averse learning algorithm to solve the problem \eqref{def:risk-averse_game}, which is  illustrated in Algorithm~\ref{alg:algorithm:FO}. 

Specifically, at episode $t$, each agent plays the action $x_{i,t}$ and samples the random seed $\xi_i^t$. 
Then, each agent collects all the previous samples $\xi_i^k$, $k=1,\ldots,t$, and computes the cost evaluations $J_i(x_t,\xi_i^k)$ and gradients $\nabla_i J_i(x_t,\xi_i^k)$, for all $k=1,\ldots,t$.
For agent $i$, we denote the CDF of the random cost $J_i(x_t,\xi_i)$ as $G_{i,t}(y)= \mathbb{P}\{ J_i(x_t,\xi_i) \leq y\}$.
With finitely many samples, the agents cannot obtain the accurate CDF but only construct the EDF $\hat{G}_{i,t}$ by
\begin{align}\label{eq:edf_FO}
    \hat{G}_{i,t}(y) = \frac{1}{t}\sum_{k=1}^{t}\textbf{1}\{ J_i(x_t,\xi_i^k) \leq y\}.
\end{align}

This distribution estimate is constructed by using all historical samples and thus each agent has $t$  samples at episode $t$. 
Using the distribution estimate $ \hat{G}_{i,t}$, each agent constructs the VaR estimate as 
\begin{align}\label{eq:FO:var:estimate:G_hat}
    \nu_{i,t} = {\rm{VaR}}_{\alpha_i} [\hat{G}_{i,t} ].
\end{align}
Using the VaR estimate $\nu_{i,t}$, we design the CVaR gradient estimate as 
\begin{align}\label{eq:FO:grad}
    g_{i,t} = \frac{1}{t \alpha_i} \sum_{k=1}^{t}  \textbf{1}\left\{ J_i(x_t,\xi_i^k)\geq \nu_{i,t}) \right\} \nabla_i J_i(x_t,\xi_i^k).
\end{align}
Then, each agent performs the following projected gradient-descent update
\begin{align}\label{eq:FO:update:action}
    x_{i,t+1} \leftarrow \mathcal{P}_{\mathcal{X}_i} ( x_{i,t} - \eta g_{i,t}).
\end{align}

\subsection{Convergence Analysis}
In this section, we provide the convergence analysis for Algorithm~\ref{alg:algorithm:FO}.

From \eqref{eq:FO:grad}, we have $\mathbb{E}\left[ g_{i,t}\right] = \nabla_{x_i} L_i(x_t,\nu_{i,t})$, which indicates that the gradient estimate \eqref{eq:FO:grad} is an unbiased estimate of $\nabla_{x_i} L_i(x_t,\nu_{i,t})$, but a biased estimate of $\nabla_i C_i(x_t)$. 
Hence, there exists a gradient estimation bias, which we define as 
\begin{align*}
    \varepsilon_{i,t}:=\nabla_{x_i} L_i(x_t,\nu_{i,t})  - \nabla_i C_i(x_t).
\end{align*}
We denote by $\nu_{i,t}^{*}$ the true VaR value of the random cost $J_i(x_t,\xi_i)$, i.e., $\nu_{i,t}^{*} ={\rm{VaR}}_{\alpha_i}[J_i(x_t,\xi_i)] $. 
By virtue of Lemma \ref{lemma:L:CVaR:gradient}, we have $\varepsilon_{i,t} =\nabla_{x_i} L_i(x_t,\nu_{i,t}) - \nabla_{x_i} L_i(x_t,\nu_{i,t}^{*})$, which indicates that the performance of CVaR gradient estimate is closely related to the VaR estimate error $|\nu_{i,t}- \nu_{i,t}^{*}|$.

It has been shown in \cite{bahadur1966note} that the quantile estimation error using finitely many samples is inversely proportional to the probability density value at the VaR point.
Therefore, we make the following assumption on the PDF of the random cost.
\begin{assumption}\label{assump:PDF}
Let $F_{i,x}(y) = \mathbb{P}\{ J_i(x,\xi_i)\leq y\}$ and $ \mathcal{Y}_{i,x} = {\rm{Range}}(J_i(x,\xi_i))$. For every $x\in \mathcal{X}$,  the distribution function $F_{i,x}$ is continuously differentiable and $L_{0}$-Lipschitz continuous. Moreover, there exists a lower bound $\underline{p}>0$ on its probability density function, i.e.,  $F'_{i,x}(y)\geq \underline{p}$ for all $ y\in\mathcal{Y}_{i,x}$.
\end{assumption}
Assumption \ref{assump:PDF} states that the PDF of the random cost $J_i(x,\xi_i)$ is both upper bounded and lower bounded. 
Note that this assumption is satisfied for some common random variables, e.g., uniform and weighted truncated random variables. 
In fact, we can construct the confidence bound for VaR estimates only when the value of the PDF at the VaR point is lower bounded, see \cite{zhang2021quantile,kolla2019concentration} for further discussions.

In what follows, we present a lemma that bounds the VaR estimation error.
\begin{lemma}\label{lemma:var_error}
Let Assumption \ref{assump:PDF} hold.
Suppose that we have $t$ samples at episode $t$.
Then, we have 
\begin{align*}
    \mathbb{P}\{ |\nu_{i,t} - \nu_{i,t}^{*}| > \epsilon \}\leq 2e^{-2 t \epsilon^2 \underline{p}^2}.
\end{align*}
\end{lemma}
\begin{proof}
The proof is motivated by \cite{shao2003mathematical}.
Let $\epsilon>0$ be fixed. Note that, for any CDF $G(y)$, $G(t)\geq t$ if and only if $y\geq G^{-1}(t)$. Hence, we have
\begin{align*}
    &\mathbb{P}\left\{ \nu_{i,t} > \nu_{i,t}^{*} + \epsilon \right\} 
    = \mathbb{P}\left\{ \hat{G}_{i,t}(\nu_{i,t}) > \hat{G}_{i,t}(\nu_{i,t}^{*} + \epsilon) \right\} \nonumber \\
    = &\mathbb{P}\left\{ G_{i,t}(\nu_{i,t}^{*} + \epsilon) -\hat{G}_{i,t}(\nu_{i,t}^{*} + \epsilon)  >  G_{i,t}(\nu_{i,t}^{*} + \epsilon) - \alpha_i\right\} \nonumber \\
    \leq & \mathbb{P}\left\{ \mathop{\rm{sup}}_{y} |\hat{G}_{i,t}(y)-G_{i,t}(y)| >  G_{i,t}(\nu_{i,t}^{*} + \epsilon) - \alpha_i\right\} \nonumber \\
    \leq & \mathbb{P}\Big\{ \mathop{\rm{sup}}_{y} |\hat{G}_{i,t}(y)-G_{i,t}(y)| >  G_{i,t}(\nu_{i,t}^{*} + \epsilon)- G_{i,t}(\nu_{i,t}^{*})\Big\} ,
\end{align*}
where the second and last equalities follow from the facts $\hat{G}_{i,t}(\nu_{i,t}) = \alpha_i $ and $G_{i,t}(\nu_{i,t}^{*})= \alpha_i$, respectively. 
By virtue of the Mean Value Theorem, there exists $\nu_1 \in (\nu_{i,t}^{*}, \nu_{i,t}^{*}+\epsilon)$ such that $G_{i,t}(\nu_{i,t}^{*} + \epsilon) - G_{i,t}(\nu_{i,t}^{*} ) = G_{i,t}^{'}(\nu_1) \epsilon$. 
Using Assumption \ref{assump:PDF}, we have that
\begin{align}
    &\mathbb{P}\left\{ \nu_{i,t} > \nu_{i,t}^{*} + \epsilon \right\} \nonumber \\
    \leq & \mathbb{P}\left\{ \mathop{\rm{sup}}_{y} |\hat{G}_{i,t}(y)-G_{i,t}(y)| >  G_{i,t}^{'}(\nu_1) \epsilon \right\} \nonumber \\
    \leq & \mathbb{P}\left\{ \mathop{\rm{sup}}_{y} |\hat{G}_{i,t}(y)-G_{i,t}(y)| >  \underline{p}  \epsilon \right\} \nonumber.
\end{align}
Using similar arguments yields the other side of the inequality. Hence, we have 
\begin{align}\label{eq:FO:num2}
    \mathbb{P}\left\{ |\nu_{i,t} - \nu_{i,t}^{*} |> \epsilon \right\}  
    \leq & \mathbb{P}\left\{ \mathop{\rm{sup}}_{y} |\hat{G}_{i,t}(y)-G_{i,t}(y)| >  \underline{p}  \epsilon \right\}.
\end{align}
By virtue of the definition of $\hat{G}_{i,t}$ in \eqref{eq:edf_FO}, this EDF is constructed using $t$ samples. Using the Dvoretzky–Kiefer–Wolfowitz (DKW) inequality, we have 
\begin{align}\label{eq:FO:num3}
    \mathbb{P}\left\{ \mathop{\rm{sup}}_{y} |\hat{G}_{i,t}(y)-G_{i,t}(y)| >  \underline{p}  \epsilon \right\}  
    \leq & 2 e^{-2 t \underline{p}^2 \epsilon^2 }.
\end{align}
Substituting \eqref{eq:FO:num3} into \eqref{eq:FO:num2} completes the proof.
\end{proof}
%

Based on Lemma \ref{lemma:var_error}, the accumulated error of the CVaR gradient estimate is bounded, which is presented in the following lemma.
\begin{lemma}\label{lemma:FO:sum:epsilon}
Given a confidence level $\bar{\gamma}$, we have that 
\begin{align*}
    \sum_{t=1}^T \left\|\varepsilon_{i.t}\right\| \leq \frac{\sqrt{2}B L_{\psi}}            { \alpha_i\underline{p}}   \sqrt{\ln \frac{2}{\bar{\gamma}}} \sqrt{T}
\end{align*}
with probability at least $1-\gamma$, where $\gamma=\bar{\gamma}T$.
\end{lemma}
\begin{proof}
    See Appendix \ref{Appen:L:FO:sum:epsilon}.
\end{proof}

For ease of notation, we define $S_1(\alpha):= \sum_i\frac{1}{\alpha_i}$ and $S_2(\alpha):=\sum_i \frac{1}{\alpha_i^2}$.
Now we are ready to present the convergence result. 
\begin{theorem}\label{theorem:FO}
Let Assumptions \ref{assump:strong_monotone}, \ref{assump:grad:bound} and \ref{assump:PDF}  hold, and select $\eta=\frac{D}{B}T^{-\frac{1}{2}} $. Then, Algorithm \ref{alg:algorithm:FO} achieves
\begin{align}
    & \frac{1}{T} \sum_{t=1}^T \mathbb{E}\left\| x_t - x^{*} \right\|^2  \nonumber \\
      = & \mathcal{O}\Big( T^{-\frac{1}{2}} \big( S_2(\alpha) + \sqrt{\ln (T / \gamma )} S_1(\alpha) \big)\Big),
\end{align}
with probability at least $1-\gamma$.
\end{theorem}
\begin{proof}
    See Appendix \ref{Appen:T:FO}.
\end{proof}


Theorem~\ref{theorem:FO} shows that Algorithm~\ref{alg:algorithm:FO} achieves a time-averaged convergence to the risk-averse Nash equilibrium with high probability. If the risk-averse game is convex and not necessarily strongly monotone, we can use similar techniques and show that Algorithm~\ref{alg:algorithm:FO} achieves sub-linear regret. 

\section{Numerical Result}\label{sec:Simulation} 
In this section, we illustrate the proposed algorithm on a Cournot game. 
Specifically, we consider two risk-averse agents $i=1,2$. 
Each agent determines the production level $x_i$ of a homogeneous product and has an individual cost function $J_i(x)=1-(2-\sum_jx_j)x_i+ 0.2x_i+\xi_i x_i$, where $\xi_i\sim U(0,1)$ is a uniform random variable. 
Here we utilize the term $\xi_i x_i$ to represent the uncertainty occurred in the market, which is proportional to the production level $x_i$. 
It is easy to verify that the game satisfies Assumption~\ref{assump:strong_monotone} and is strongly monotone. 
The agents have their own risk levels $\alpha_i$ and they aim to minimize the CVaR value of their cost functions. We select $\alpha_1 = 0.4$ and $\alpha_2=0.8$.  
It can be verified that the Nash equilibrium point of the formulated risk-averse game is $x^{*}=(0.2667,0.4667)$.

To compute the distribution function in Algorithm~\ref{alg:algorithm:FO} in practice, we partition the interval $[0, U]$ into 1000 equal-width bins and approximate the expectation by the sum of finite terms. 
We compare Algorithm~\ref{alg:algorithm:FO} with the zeroth-order algorithm in \cite{wang2022zeroth}, which we term  the zeroth-order algorithm with momentum. 
To explore the effect of bias in the estimates of the VaR values, we run another first-order algorithm with accurate VaR values at each iteration, which we term the unbiased first-order algorithm. 
Each algorithm is run for 20 trials and the parameters of these algorithms are separately optimally tuned. 
The convergence results are presented in Figure \ref{fig_algs}.
We observe that our first-order method, i.e., Algorithm~\ref{alg:algorithm:FO}, outperforms the zeroth-order algorithm. 
Since Algorithm~\ref{alg:algorithm:FO} cannot entirely eliminate the bias in the VaR value estimates, it performs worse than the unbiased first-order algorithm. However, as the number of samples increases and the VaR estimate error diminishes, the performance of these two algorithms becomes close.

\begin{figure}[t] 
\begin{center}
\centerline{\includegraphics[width=1\columnwidth]{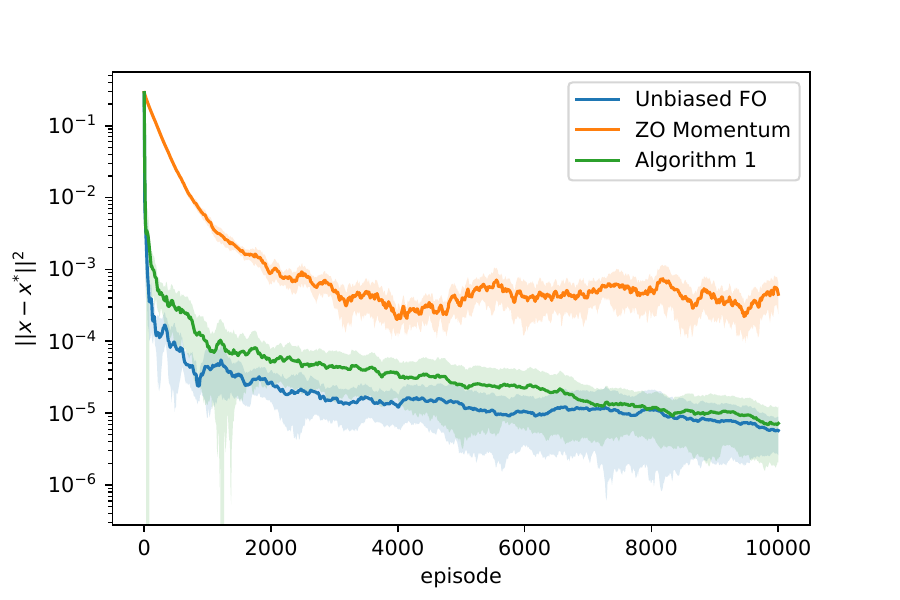}}
\caption{Error to the Nash equilibrium among Algorithm 1, unbiased first-order algorithm (Unbiased FO), and the zeroth-order algorithm with momentum (ZO momentum) in \cite{wang2022risk}. The solid lines and shades are averages and standard deviations over 20 runs.}
\label{fig_algs}
\end{center}
\vskip -0.2in
\end{figure}

\section{Conclusion}\label{sec:Conclusion} 
In this work, we proposed a first-order method to solve convex games with risk-averse agents.
We showed that the VaR estimates play a key role in estimating the CVaR gradient.
Assuming that the game is strongly monotone,  we showed that the VaR estimates can be bounded and, as a result, our proposed algorithm converges to the Nash equilibrium with high probability.
We provided numerical simulations to illustrate the performance of our algorithm.

\begin{appendix}


\subsection{Proof of Lemma \ref{lemma:FO:sum:epsilon}}\label{Appen:L:FO:sum:epsilon}
By applying Lemma~\ref{lemma:var_error} and setting $\bar{\gamma} = 2e^{-2 t \epsilon^2 \underline{p}^2}$ with $\bar{\gamma} = \gamma / T$, we have 
\begin{align}\label{eq:event:var_estimate}
    \mathbb{P}\left\{ |\nu_{i,t} - \nu_{i,t}^{*}| > \frac{1}{\underline{p}\sqrt{2 t}} \sqrt{\ln \frac{2C}{\bar{\gamma}}} \right\}\leq \bar{\gamma}.
\end{align}
Define the events in \eqref{eq:event:var_estimate} as $\mathcal{B}_t$.
Then, for all $t=1,\ldots, T$, we have
$ |\nu_{i,t} - \nu_{i,t}^{*}| \leq \frac{1}{\underline{p}\sqrt{2t}} \sqrt{\ln \frac{2}{\bar{\gamma}}}, \; \forall t=1,\ldots, T, $
with probability at least $1-\gamma$, since $1-\mathbb{P} \{ \bigcup_{t=1}^T \mathcal{B}_t\} \geq 1- \sum_{t=1}^T\mathbb{P}\{\mathcal{B}_t\}\geq 1-T \frac{\gamma}{T}\geq1-\gamma$.

Next, we analyze the property of $\varepsilon_{i,t}$. Set $\nu_m = {\rm{min}}\{ \nu_{i,t},\nu_{i,t}^{*}\}$, $\nu_M ={\rm{ max}}\{ \nu_{i,t},\nu_{i,t}^{*}\}$. We have
\begin{align*}
    &\varepsilon_{i,t} = \nabla_{x_i} L_i(x_t,\nu_{i,t}) -  \nabla_i C_i(x_t) \nonumber \\
    =& \mathbb{E} \Big[\frac{1}{\alpha_i}  \textbf{1}\left\{ J_i(x_t,\xi_i)\geq \nu_{i,t}) \right\} \nabla_i J_i(x_t,\xi_i)\Big] \nonumber \\
    &-\mathbb{E} \Big[\frac{1}{\alpha_i}  \textbf{1}\left\{ J_i(x_t,\xi_i)\geq \nu_{i,t}^{*}) \right\} \nabla_i J_i(x_t,\xi_i)\Big] \nonumber \\
    = & \mathbb{E} \Big[\frac{1}{\alpha_i}  {\rm{sgn}}\{\nu_{i,t}-\nu_{i,t}^{*}\} \textbf{1}\left\{\nu_m \leq J_i(x,\xi_i)\leq \nu_M \right\}    \nabla_i J_i(x_t,\xi_i)\Big]. \nonumber
\end{align*}
Using the fact that $\left\|\mathbb{E}[X Y]\right\| \leq \mathbb{E}[\left\|X\right\| \left\|Y\right\|]$ for any random variable $X,Y$,  for all $ t\geq 1$, we have 
\begin{align}\label{eq:FO:ineq:1}
    \left\|\varepsilon_{i,t}\right\| \leq & \frac{1}{\alpha_i} \mathbb{E}[ \textbf{1}\left\{\nu_m \leq J_i(x,\xi_i)\leq \nu_M) \right\}] B \nonumber \\
    = & \frac{B}{\alpha_i} (G_{i,t}(\nu_M) - G_{i,t}(\nu_m))  \nonumber \\
    \leq & \frac{B L_{\psi}} {\alpha_i} |\nu_{i,t} - \nu_{i,t}^{*}| \nonumber \\
    \leq &  \frac{B L_{\psi}}{ \alpha_i\underline{p} \sqrt{2t}} \sqrt{\ln \frac{2T}{\gamma}} ,
\end{align}
with probability at least $1-\gamma$.
Summing up the inequality \eqref{eq:FO:ineq:1} over $t=1,\ldots,T$, we have
\begin{align*}
    \sum_{t=1}^T \left\|\varepsilon_{i.t}\right\| 
    \leq & \sum_{t=1}^T \frac{B L_{\psi}}{ \alpha_i\underline{p}\sqrt{2t}} \sqrt{\ln \frac{2T}{\gamma}} \nonumber \\
    \leq & \frac{B L_{\psi}}{ \alpha_i\underline{p}\sqrt{2}} \sqrt{\ln \frac{2T}{\gamma}} \Big( 1+ \int_{1}^T \frac{1}{\sqrt{t}} dt \Big)\nonumber \\
    \leq & \frac{B L_{\psi}}{ \alpha_i\underline{p}\sqrt{2}} \sqrt{\ln \frac{2T}{\gamma}} \Big( 1+ 2\sqrt{t} \Big|_{1}^T dt \Big)\nonumber \\
    \leq & \frac{\sqrt{2}B L_{\psi}}{ \alpha_i\underline{p}} \sqrt{\ln \frac{2T}{\gamma}} \sqrt{T}, \nonumber 
\end{align*}
which completes the proof.

\subsection{Proof of Theorem \ref{theorem:FO}}\label{Appen:T:FO}

From the update equation \eqref{eq:FO:update:action}, we have
\begin{align}\label{eq:FO:t1}
    &\quad \left\| x_{i,t+1} - x_i^{*}\right\|^2 \nonumber \\
    &= \left\| \mathcal{P}_{\mathcal{X}_i}(x_{i,t}- \eta g_{i,t} ) - x_i^{*}\right\|^2 \nonumber \\
    &\leq \left\| x_{i,t} - x_i^{*} - \eta g_{i,t} \right\|^2 \nonumber \\
    & = \left\| x_{i,t} - x_i^{*}\right\|^2 + \eta^2 \left\| g_{i,t} \right\|^2 - 2 \eta \langle g_{i,t}, x_{i,t}-x_i^{*} \rangle,
\end{align}
where the inequality holds since the projection operator is non-expansive.
Taking expectation on both sides of the inequality \eqref{eq:FO:t1}, it follows that 
\begin{align}\label{eq:FO:t2}
    & \mathbb{E} \left[ \left\| x_{i,t+1} - x_i^{*}\right\|^2 \right] \nonumber \\
    \leq &\mathbb{E} \left[\left\| x_{i,t} - x_i^{*}\right\|^2\right] +  \eta^2  \mathbb{E}\left[ \left\| g_{i,t} \right\|^2\right] \nonumber \\
    &- 2\eta \langle \nabla_{x_i} L_i(x_t,\nu_{i,t}), x_{i,t}-x_i^{*} \rangle \nonumber \\
    \leq &\mathbb{E} \left[\left\| x_{i,t} - x_i^{*}\right\|^2\right] +  \eta^2   \frac{B^2}{\alpha_i^2} \nonumber \\
    &- 2\eta \langle \mathbb{E} \left[\varepsilon_{i,t} \right]+\nabla_i C_i(x_t), x_{i,t}-x_i^{*}\rangle .
\end{align}
Since $x^{*}$ is a NE of the risk-averse game, we have that $ \langle  \nabla_i C_i(x^{*}), x_{i,t}-x_i^{*}\rangle \geq 0$ for all $x_i \in \mathcal{X}_{i}$. Summing the inequality \eqref{eq:FO:t2} over $i=1,\ldots,N$, we have
\begin{align}\label{eq:FO:t3}
    &\mathbb{E} \left[\left\| x_{t+1} - x^{*}\right\|^2 \right]\nonumber \\
    \leq & \mathbb{E} \left[ \left\| x_{t} - x^{*}\right\|^2\right] +  \eta^2  B^2 \sum_i \frac{1}{\alpha_i^2}   \nonumber \\
    & - 2\eta \sum_i \langle \mathbb{E}\left[ \varepsilon_{i,t} \right]+\nabla_i C_i(x_t), x_{i,t}-x_i^{*}\rangle \nonumber \\
    \leq  &\mathbb{E} \left\| x_{t} - x^{*}\right\|^2 +   \sum_i \frac{\eta^2  B^2}{\alpha_i^2}    - 2\eta\sum_i \langle  \mathbb{E}\left[\varepsilon_{i,t}\right] , x_{i,t}-x_i^{*}\rangle \nonumber \\
    & - 2\eta \sum_i \langle  \nabla_i C_i(x_t)- \nabla_i C_i(x^{*}), x_{i,t}-x_i^{*}\rangle \nonumber \\
     \leq  &  (1 - 2\eta m) \mathbb{E}  \left\| x_{t} - x^{*}\right\|^2 +  \eta^2  B^2 \sum_i \frac{1}{\alpha_i^2} \nonumber \\
     &+ 2\eta D \sum_i \mathbb{E}\left[\left\| \varepsilon_{i,t} \right\| \right].
\end{align}
Rearranging the terms in \eqref{eq:FO:t3} and summing up over $t=1,\ldots,T$, we have
\begin{align*}
    \frac{1}{T} \sum_{t=1}^T \mathbb{E} \left\| x_{t} - x^{*}\right\|^2 \leq  \frac{D^2}{2\eta m T} + \frac{\eta B^2 S_2(\alpha)}{2m}& \nonumber \\
     + \frac{D}{mT}\sum_{t=1}^T \sum_i \mathbb{E}\left[\left\| \varepsilon_{i,t} \right\|\right]&.
\end{align*}
Applying the results in Lemma \ref{lemma:FO:sum:epsilon}, it holds with probability at least $1-\gamma$ that
\begin{align}\label{eq:FO:t5}
    \frac{1}{T} \sum_{t=1}^T \mathbb{E} \left\| x_{t} - x^{*}\right\|^2 
    \leq  \frac{D^2}{2\eta m T} + \frac{\eta B^2 S_2(\alpha)}{2m} & \nonumber \\
     + \frac{\sqrt{2}D B L_{\psi} S_1(\alpha) }{ m  \underline{p}}  \sqrt{\ln \frac{2T}{\gamma}} T^{-\frac{1}{2}} &.
\end{align}
Substituting $\eta = \frac{D}{B T^{\frac{1}{2}}}$ into \eqref{eq:FO:t5} yields the desired result. The proof ends.

\end{appendix}





%



\bibliography{00_citation}
\bibliographystyle{unsrt}







\end{document}